\documentclass[11pt]{amsart}
\usepackage{amsfonts, amsbsy, amsmath, amsthm, amssymb, latexsym, verbatim, enumerate,hyperref}
\usepackage{mathrsfs}
\usepackage[top=30mm,right=30mm,bottom=30mm,left=30mm]{geometry}

\usepackage{bm}
\usepackage[all]{xy}

\usepackage[pdftex]{color,graphicx}

\makeatletter
\newcommand{\imod}[1]{\allowbreak\mkern4mu({\operator@font mod}\,\,#1)}
\makeatother

\newtheorem{propo}{Proposition}

\newtheorem{lemma}[propo]{Lemma}

\newtheorem{corollary}[propo]{Corollary}
\newtheorem{theor}[propo]{Theorem}

\newcommand{\Ker}{\operatorname{Ker}}
\newcommand{\Aut}{{\mathrm {Aut}}}

\newcommand{\Out}{{\mathrm {Out}}}

\newcommand{\Irr}{{\mathrm {Irr}}}

\newcommand{\Syl}{{\mathrm {Syl}}}

\newcommand{\ZZ}{{\mathbb Z}}

\newcommand{\FF}{{\mathbb F}}

\newcommand{\GC}{{\underline{\boldsymbol{G}}}}

\newcommand{\ZB}{\mathbf{Z}}
\newcommand{\CB}{\mathbf{C}}
\newcommand{\NB}{\mathbf{N}}
\newcommand{\OB}{\mathbf{O}}

\newcommand{\lam}{\lambda}

\newcommand{\al}{\alpha}

\newcommand{\GL}{\mathrm {GL}}
\newcommand{\SL}{\mathrm {SL}}
\newcommand{\PSL}{\mathrm {PSL}}

\newcommand{\GU}{\mathrm {GU}}
\newcommand{\SU}{\mathrm {SU}}

\newcommand{\Sp}{\mathrm {Sp}}

\newcommand{\SO}{\mathrm {SO}}
\newcommand{\GO}{\mathrm {GO}}

\newcommand{\Spin}{{\mathrm{Spin}}}

\newcommand{\tw}[1]{{}^#1\!}

\newcommand{\St}{\mathsf {St}}

\begin{document}

\title{An extension of Gow's theorem}

\author{Robert M. Guralnick}
\address{Department of Mathematics, University of Southern California,
Los Angeles, CA 90089-2532, USA}
\email{guralnic@math.usc.edu}

\author{Pham Huu Tiep}
\address{Department of Mathematics\\ 
Rutgers University\\ Piscataway, NJ 08854\\    
U. S. A.} 
\email{tiep@math.rutgers.edu}


\thanks{The first author was partially supported by the NSF
(grant DMS-1901595) and a Simons Foundation Fellowship 609771. The second author was partially supported by the NSF (grant
DMS-2200850) and the Joshua Barlaz Chair in Mathematics.}

\dedicatory{In memory of our good friend and esteemed colleague Gary Seitz} 

\begin{abstract}
We extend Gow's theorem to finite groups $G$ whose generalized Fitting subgroup is $\ZB(G)S$ for a quasisimple Lie-type group $S$ of simply connected type in characteristic $p$, and whose center $\ZB(G)$ has $p'$-order.
\end{abstract}

\maketitle

A result of Rod Gow \cite[Theorem 2]{Gow} asserts that the product $a^Gb^G$ of any two regular semisimple classes in a finite simple group 
of Lie type $G$ contains every nontrivial semisimple element $x \in G$. This result has been used in many applications. It has also been 
extended to any quasisimple Lie-type group of 
simply connected type: the product $a^Gb^G$ of any two regular semisimple classes in $G$ contains every non-central semisimple element $x \in G$, 
see \cite[Lemma 5.1]{GT}.

In this paper we will further extend Gow's theorem. Let $p$ be a prime and let $\GC$ be a simple, simple connected algebraic group defined over 
$\overline{\FF}_p$. Let $F:\GC \to \GC$ be a Steinberg endomorphism, so that 
$$S:=\GC^F$$
be quasisimple. (In particular, we do not view $\PSL_2(9)$ as $\Sp_4(2)'$, $\SU_3(3)$ as $G_2(2)'$, and 
$\SL_2(8)$ as $\tw2 G_2(3)'$.) 

We will consider finite groups $G$ with
\begin{equation}\label{eq10}
  F^*(G)=\ZB(G)S \mbox{ and }p \nmid |\ZB(G)|,
\end{equation}  
(so $\CB_G(S) = \CB_G(F^*(G))=\ZB(G)\ZB(S)$ is a $p'$-group), 
and aim to show that the product $a^Gb^G$ of certain two conjugacy classes in $G$ will cover all elements $g \in G$ of certain kind. 
Before going on we state a special case of a consequence of our main result which is less technical.  Versions of this result have already been used in
\cite{AGS, GTT}.

\begin{corollary} \label{commutator}   Let $G$ be as above.   Assume that  $a \in G$ has order prime to $p$,  $|\CB_S(a)|$ has order prime to $p$ and that 
$s \in S \smallsetminus \ZB(S)$ is semisimple.
Then $s^t = [a,u]$ for some $t, u \in S$.
\end{corollary}

In fact, one need not assume that $g$ has order prime to $p$.   See  our main result Theorem \ref{main1}.  

Setting $G_1 :=  \langle S,a,b \rangle$, we have that $S \lhd G_1$, and hence $S \lhd F^*(G_1)$. By \eqref{eq10}, $F(G_1)$, as well as any other than $S$ 
semisimple normal subgroup of $G_1$, are contained in 
$$\ZB(G)\ZB(S) \cap G_1 \leq \ZB(G_1)S$$ 
as they centralize $S$. It follows that 
$F^*(G_1) = \ZB(G_1)S$. Furthermore, 
$$\ZB(G_1) \leq \CB_G(S)=\ZB(G)\ZB(S)$$ 
is also a $p'$-subgroup. Hence, 
for our purposes, we may assume 
\begin{equation}\label{eq11}
  G = \langle S,a,b \rangle.
\end{equation}   

Let $\St$ denote the Steinberg character of $S$. By \cite[Corollary D]{F1}, $\St$ extends to a rational-valued character $\St_G$ of $G$ (called 
the {\it basic $p$-Steinberg character} of $G$). Moreover,
by \cite[Theorem C]{F2}, there is a Sylow $p$-subgroup $P$ and a $p$-subgroup $D$ of $G$, of order
$$|D|=p^d=|G/S|_p,$$
such that $P = Q \rtimes D$ for a Sylow $p$-subgroup 
$Q$ of $S$ and the following statement holds. For any element $x \in G$, $\St_G(x) \neq 0$ if and only if $x_p \in D$ (up to conjugation)
for the $p$-part of $x$, in which case
$$\St_G(x) = \pm |\CB_S(x)|_p.$$
In view of these results, the proper generalization to $G$ of regular semisimple classes in $S$ will be that $a,b \in G$ satisfy
\begin{equation}\label{for-a}
  a_p \in D \mbox{ up to conjugation in }G, \mbox{ and }p \nmid |\CB_S(a)|,
\end{equation} 
and
\begin{equation}\label{for-b}
  b_p \in D \mbox{ up to conjugation in }G, \mbox{ and }p \nmid |\CB_S(b)|,
\end{equation} 
Certainly, $g \in G$ can belong to $a^Gb^G$ only when it does so in the solvable group $G/S$, so we will assume
\begin{equation}\label{for-g}
  gS \in (aS)^{G/S}(bS)^{G/S}, \mbox{ and }g_p \in D \mbox{ up to conjugation in }G.
\end{equation}
For instance, if $G/S$ is abelian, then the first condition in \eqref{for-g} is equivalent to $g \in abS$.  

\begin{propo}\label{cent}
The following statements hold.
\begin{enumerate}[\rm(i)]
\item If $p > 2$ and $Q \in \Syl_p(S)$, then $\CB_G(Q) = \ZB(G)\ZB(S)\ZB(Q)$.
\item If $g \in G \smallsetminus \ZB(G)\ZB(S)$ and $g_p \in D$, then $p$ divides $[S:\CB_S(g)]$. 
\end{enumerate}
\end{propo}

\begin{proof}
Since $Z:=\ZB(G)\ZB(S)$ is a $p'$-group centralizing $Q$ and normal in $G$, we may work in $\bar{G}:=G/Z$ and identify 
$Q$ with $\bar{Q}:=QZ/Z$. Note that $Z \cap S = \ZB(S)$. Moreover, as $S$ is perfect, we have
$\CB_G(S/\ZB(S)) = \CB_G(S) = Z$. It follows that 
$$\bar{S}:=S/\ZB(S) \lhd \bar{G} \leq \Aut(\bar{S}),$$ 
i.e. $\bar{G}$ is almost simple. 

Consider any element $x \in \bar{C}:=\CB_{\bar{G}}(\bar{Q})$. Then $H:= \langle \bar{S},x \rangle \leq \bar{G}$ is also almost simple,
whence $\OB_{p'}(H)=1$, and $R:=\langle \bar{Q},x_p \rangle$ is a Sylow $p$-subgroup of $H$ centralized by $x$.  It follows 
that $R=\OB_p(\NB_H(R))$. By \cite[Corollary 3.1.4]{GLS}, $\OB_p(\NB_H(R)) = F^*(\NB_H(R))$, whence $x$ belongs to
$\CB_{\NB_H(R)}(R) \leq R$ and so $x=x_p$ is a $p$-element. Thus $\bar{C}$ is a $p$-group.

Similarly, $\bar{Q}=\OB_p(\NB_{\bar{S}}(\bar{Q})) = F^*(\NB_{\bar{S}}(\bar{Q}))$, and so
$$\bar{C} \cap \bar{S} =\CB_{\bar{S}}(\bar{Q}) = \CB_{\NB_{\bar{S}}(\bar{Q})}(\bar{Q}) \leq \bar{Q}.$$
It follows that $\bar{C} \cap \bar{S} = \ZB(\bar{Q})$.

\smallskip
(i) Now we assume $p > 2$ and show that $\bar{C} \leq \bar{S}$, which implies that $\CB_G(Q)=\ZB(Q)Z$.
Assume the contrary: $\bar{C} \not\leq \bar{S}$. Since $\bar{C}$ is a $p$-group, we can find 
a $p$-element $x \in \bar{C} \smallsetminus \bar{S}$; in particular, $[x,\bar{Q}]=1$. Now $R:=\langle x,\bar{Q} \rangle$ is 
Sylow $p$-subgroup of $H := \langle \bar{S},x \rangle$ and $x \in \ZB(R)$. As $\bar{S} \lhd H \leq \bar{G} \leq \Aut(\bar{S})$, we still have 
$\OB_{p'}(H)=1$. Now, if $p > 2$ then $\ZB(R) \leq F^*(H) = \bar{S}$ by \cite[Corollary 1.2]{GGLN}, and hence $x \in \bar{S}$, contrary to
the choice of $x$.

\smallskip
(ii) Assume the contrary that $p \nmid [S:\CB_S(g)]$. Conjugating $g$ suitably, we may assume that $g \in \CB_G(Q)$ with
$Q \in \Syl_p(S)$ as before.

Suppose first that $p>2$. Then $g \in \ZB(G)\ZB(S)\ZB(Q)$ by (i), and so $g_p \in S$. But $g_p$ is conjugate to an element in $D$ by assumption and
$D \cap S = 1$, so $g_p = 1$. It follows that $g \in \ZB(G)\ZB(S)$, a contradiction.

Thus we have $p=2$. Then $\St_G(1)=|Q|=|\CB_S(g)|_p =  \pm \St_G(g)$. 
On the other hand, $\St$ is trivial at $\ZB(S)$, 
so the generalized center of $\St_G$ contains $Z=\ZB(G)\ZB(S)$ and hence equals $G$ as $G/Z$ is almost simple with socle $\bar{S}$.
As the generalized center of $\St_G$ contains $g$, we conclude that $g \in Z$, again a contradiction.
\end{proof}

Fix any element $g \in G$ satisfying \eqref{for-g}. Then $S\CB_G(g) \leq G$, so
\begin{equation}\label{ratio}
  \ZZ \ni [G:S\CB_G(g)] = \frac{|G| \cdot |\CB_S(g)|}{|S| \cdot |\CB_G(g)|} = \frac{[G:\CB_G(g)]}{[S:\CB_S(g)]}.
\end{equation}  
Write
\begin{equation}\label{ratio2}
  \frac{[G:\CB_G(g)]_p}{[S:\CB_S(g)]_p} = p^e.
\end{equation}   

\begin{lemma}\label{xyz}
Let $X$ be a finite group, which is abelian-by-cyclic, that is, 
$X$ has a normal abelian subgroup $A \lhd X$ such that $X/A$ is cyclic. Suppose $x,y,z \in X$ are such that
$$X = \langle x,y\rangle \mbox{ and }z \equiv xy \pmod{[X,X]}.$$
Then 
$$\sum_{\alpha \in \Irr(X)}\frac{\alpha(x)\alpha(y)\overline{\alpha(z)}}{\al(1)} = |X/[X,X]|.$$ 
\end{lemma}

\begin{proof}
The condition $z \equiv xy \pmod{[X,X]}$ implies that 
$$\sum_{\alpha \in \Irr(X),~\al(1)=1}\frac{\alpha(x)\alpha(y)\overline{\alpha(z)}}{\al(1)} = |X/[X,X]|.$$
Hence it suffices to show that the contribution of any non-linear $\al \in \Irr(X)$ to the sum in the statement is $0$. Consider any irreducible constituent 
$\lam$ of $\al|_A$. Suppose $\lam$ is not $X$-invariant. As $X = \langle x,y \rangle$, we may assume that $\lam$ is not $x$-invariant, in which case 
$\al(x)=0$ by Clifford's theorem and the contribution is $0$ as claimed. 

Suppose now that $\lam$ is $X$-invariant. Then for any $a \in A$ and $t \in X$, as $\lam(1)=1$ we have
$$\lam(tat^{-1}a^{-1})= \lam(tat^{-1})/\lam(a) = 1,$$
whence $[t,a] \in \Ker(\lam)$ and $\Ker(\lam) \lhd X$. It follows that $A/\Ker(\lam) \leq \ZB(X/\Ker(\lam))$. But $X/A$ is cyclic, 
so $X/\Ker(\lam)$ is abelian. Now $\lam$ is the unique irreducible constituent of $\al_A$, so $\Ker(\lam) \leq \Ker(\al)$, and 
hence $\al$, viewed as an irreducible character of $X/\Ker(\lam)$, must be linear, contrary to the assumption $\al(1)>1$.   
\end{proof}

\begin{propo}\label{sum-1}
Under the assumptions \eqref{eq10}--\eqref{for-g}, assume in addition that $G/S$ is abelian-by-cyclic. Then 
$$\Sigma_1:=\sum_{\chi \in \Irr(G|\St)}\frac{\chi(a)\chi(b)\overline{\chi(g)}}{\chi(1)}\cdot |g^G|$$
is a rational integer whose $p$-part is at most $p^{d+e}$.
\end{propo}

\begin{proof}
As mentioned above, $\St$ extends to $\St_G$. Hence, by Gallagher's theorem \cite[(6.17)]{Is}, any $\chi \in \Irr(G|\St)$ is 
of the form 
$$\chi=\St_G\al$$ 
with $\al \in \Irr(G/S)$. Using \eqref{eq11}, \eqref{for-g} and Lemma \ref{xyz}, we see that
$$\sum_{\alpha \in \Irr(G/S)}\frac{\alpha(a)\alpha(b)\overline{\alpha(g)}}{\al(1)}$$
is a rational integer whose $p$-part is at most $p^d$. 

On the other hand, by \eqref{for-a} and \eqref{for-b} we see that
$$\frac{\St_G(a)\St_G(b)\overline{\St_G(g)}}{\St_G(1)}\cdot |g^G| 
    = \pm \frac{|\CB_S(g)|_p\cdot |G|_p \cdot |G|_{p'}}{|S|_p \cdot |\CB_G(g)|_p \cdot |\CB_G(g)|_{p'}}
    = \pm \frac{[G:\CB_G(g)]_p}{[S:\CB_S(g)]_p} \cdot [G:\CB_G(g)]_{p'}$$
is $p^e$ times a $p'$-integer. Hence the statement follows.    
\end{proof}

Recall that $\St$ is the only $p$-defect zero character of $S$. By the main result of \cite{Hu}, all the remaining characters of 
$S$ belong to $p$-blocks of maximal defect. The next result deals with these characters. 

\begin{propo}\label{sum-2}
Under the assumptions \eqref{eq10}--\eqref{for-g}, assume in addition that $G/S$ has a cyclic Sylow $p$-subgroup and a normal 
$p$-complement, and that $g \notin \ZB(G)\ZB(S)$. Then 
$$\Sigma_2:=\sum_{\chi \in \Irr(G) \smallsetminus \Irr(G|\St)}\frac{\chi(a)\chi(b)\overline{\chi(g)}}{\chi(1)}\cdot |g^G|$$
is $p^{d+e+1}$ times an algebraic integer.
\end{propo}

\begin{proof}
By the hypothesis we can write $G/S = (H/S) \rtimes D$ for some normal subgroup $H \geq S$ of $G$. Note that
any $\chi  \in \Irr(G) \smallsetminus \Irr(G|\St)$ lies above some $\theta \in \Irr(H)$ which does not lie above $\St$. 
Suppose $\theta$ is not $G$-invariant. As 
$G = \langle H,a,b \rangle$, we may assume that $\theta$ is not $a$-invariant, in which case 
$\chi(a)=0$ by Clifford's theorem, and the contribution of $\chi$ to $\Sigma_2$ is $0$ as claimed. 

Hence we need to count the total contribution to $\Sigma_2$ of the characters $\chi \in \Irr(G|\theta)$, where
$\theta \notin \Irr(H|\St)$ is $G$-invariant. Since $G/H$ is cyclic, any such $\theta$ extends to a character $\chi_1$ of $G$, and we 
may write 
$$\Irr(G|\theta) = \{\chi_1\mu \mid \mu \in \Irr(G/H)\}.$$
By the assumption $\theta \notin \Irr(H|\St)$, every irreducible constituent of $\theta|_S$ belongs to an $S$-block $B_S$ of maximal $p$-defect. 

Conjugating $g$ suitably, we may assume that $g_p \in D$. Note that $g_{p'} \in H$, so $g=g_pg_{p'}$ belongs to 
$$K:= \langle H,g_p \rangle \lhd G.$$
Set 
$$\chi_2 := (\chi_1)|_K \in \Irr(K|\theta).$$
Now the $p$-block $B$ of $H$ that contains $\theta$ covers $B_S$, and $p \nmid |H/S|$, so $B$ has maximal defect, see e.g. \cite[Theorem 9.26]{N}.
But $K/H \hookrightarrow D$ is a $p$-group, so by \cite[Corollary 9.6]{N} there is a unique $p$-block $B_2$ of $K$ that covers $B$. In particular, 
$$\Irr(K|\theta) \subseteq \Irr(B_2).$$
Moreover, $B$ is $K$-invariant as $\theta$ is $K$-invariant, whence $B_2$ is of maximal defect by \cite[Theorem 9.17]{N}. It follows that 
$B_2$ contains a character $\chi_0$ of height zero, and so of $p'$-degree.

As $\chi_0$ and $\chi_2$ belong to the same block, we know that the two algebraic integers 
$$\omega_{\chi_i}(g) = \frac{\chi_i(g)}{\chi_i(1)} \cdot |g^K|$$ 
for $i \in \{0,2\}$ are congruent modulo $p$. 
By Proposition \ref{cent}, $|g^S|$ is divisible by $p$, so 
$|g^K|$ is divisible by $p$ as well, see the computation in \eqref{ratio}. But $p \nmid \chi_0(1)$, so $p|\omega_{\chi_0}(g)$. It follows that 
\begin{equation}\label{ratio3}
  p \mbox{ divides }\omega_{\chi_2}(g) = \frac{\chi_2(g)}{\chi_2(1)} \cdot |g^K| = \frac{\chi_1(g)}{\chi_1(1)} \cdot |g^K|.
\end{equation}   
Next, \eqref{ratio} applied to $S \lhd H$ with $p \nmid |H/S|$ shows that 
$$|g^S|_p = |g^H|_p.$$
On the other hand, $g_p$ centralizes $g$, and $K=\langle H,g_p \rangle$, so
$H\CB_K(g) = K$, showing that $g^K= g^H$. Hence
$|g^S|_p=|g^K|_p$, and \eqref{ratio2} becomes
$$p^e = \frac{|g^G|_p}{|g^K|_p}.$$ 
Together with \eqref{ratio3}, we now obtain
$$p^{e+1} \mbox{ divides }\omega_{\chi_1}(g)=\frac{\chi_1(g)}{\chi_1(1)} \cdot |g^G|.$$

Now, \eqref{for-g} implies that $g \equiv ab \pmod{G/H}$, and so 
$$\sum_{\mu \in \Irr(G/H)}\frac{\mu(a)\mu(b)\overline{\mu(g)}}{\mu(1)}= |G/H| = p^d.$$
It follows that
$$\sum_{\chi \in \Irr(G|\theta)}\frac{\chi(a)\chi(b)\overline{\chi(g)}}{\chi(1)}\cdot |g^G| = 
    \omega_{\chi_1}(g)\sum_{\mu \in \Irr(G/H)}\frac{\mu(a)\mu(b)\overline{\mu(g)}}{\mu(1)}=p^d\omega_{\chi_1}(g)$$
is $p^{d+e+1}$ times an algebraic integer.    
\end{proof}

\begin{theor}\label{main1}
Under the assumptions \eqref{eq10}--\eqref{for-g}, assume in addition that all the following conditions hold.
\begin{enumerate}[\rm(a)]
\item $G/S$ is abelian-by-cyclic. 
\item $G/S$ has cyclic Sylow $p$-subgroups and a normal $p$-complement.
\item $g \notin \ZB(G)\ZB(S)$.
\end{enumerate}
Then $g \in a^Gb^G$.
\end{theor}

\begin{proof}
By Propositions \ref{sum-1} and \ref{sum-2}, 
$$|g^G|\sum_{\chi \in \Irr(G)}\frac{\chi(a)\chi(b)\overline{\chi(g)}}{\chi(1)}=\Sigma_1+\Sigma_2$$
is $p^s(u+p^tv)$, where $u \in \ZZ\smallsetminus p\ZZ$, $v$ is an algebraic integer, $0 \leq s \leq d+e$, and $t \geq 1$. Now if 
$u+p^tv=0$, then $v=-u/p^t$ is rational and an algebraic integer, so $v \in \ZZ$ and $u \in p\ZZ$, a contradiction.
Thus  
$$\sum_{\chi \in \Irr(G)}\frac{\chi(a)\chi(b)\overline{\chi(g)}}{\chi(1)} \neq 0,$$
and so $g \in a^Gb^G$ by Frobenius character formula.
\end{proof}

\begin{proof}[Proof of Corollary \ref{commutator}]
To prove the result, we may replace $G$ by $\langle S,a,b\rangle$ with $b:=a^{-1}$. Then \eqref{for-a}--\eqref{for-b} hold with $g:=s$.
Now $G/S$ is cyclic, so by $s \in a^Gb^G$. But $a^G=a^S$ and $b^G = b^S$ since $G = \langle S,a \rangle = \langle S,b \rangle$, 
so the statement follows.
\end{proof}

In what follows, $q$ is always a power of the prime $p$. We will use the structure of $\Aut(\bar{S})$ as described in \cite[Theorem 2.5.12]{GLS}, in 
particular the notations $\mathrm{Inndiag}(\bar{S})$ and $\mathrm{Outdiag}(\bar{S})$.

\begin{theor}\label{main2}
Under the assumptions \eqref{eq10}--\eqref{for-g}, assume in addition that all the following conditions hold for $g$, $\bar{S}=S/\ZB(S)$,
and $\bar{G}=G/\ZB(G)\ZB(S)$.
\begin{enumerate}[\rm(a)]
\item $g \notin \ZB(G)\ZB(S)$.
\item If $\bar{S}=\PSL_n(q)$ with $n \geq 3$, or $\bar{S}=P\Omega^+_{2n}(q)$ with $n \geq 4$, or 
$S=E_6(q)$, then the quotient $\bar{G}/\bigl(\bar{G} \cap \mathrm{Inndiag}(\bar{S})\bigr)$ is cyclic.
\end{enumerate}
Then $g \in a^Gb^G$.
\end{theor}

\begin{proof}
Recall that $G/S \cong \bar{G}/\bar{S}$ is a subgroup of $O:=\Out(\bar{S})$. By 
Theorem \ref{main1}, we need to show that $A=G/S$ satisfies both of the conditions (a) and (b) listed therein. Note that both (a) 
and (b) in Theorem \ref{main1} follow from the condition
\begin{equation}\label{cond-a}
  A \mbox{ admits a normal abelian }p'\mbox{-subgroup }B \mbox{ with }A/B \mbox{ being cyclic}.
\end{equation}  
In turn, \eqref{cond-a} is a consequence of the condition
\begin{equation}\label{cond-o}
  O:=\Out(\bar{S}) \mbox{ admits a normal abelian }p'\mbox{-subgroup }J \mbox{ with }O/J \mbox{ being cyclic}.
\end{equation} 
(Indeed, taking $B:=A \cap J$ we have $A/B \hookrightarrow O/J$.) 

\smallskip
Set $J:=\mathrm{Outdiag}(\bar{S}):=\mathrm{Inndiag}(\bar{S})/\bar{S}$. Now, if $\bar{S}$ is a {\it twisted} group, i.e. the parameter $d$ for $\bar{S} \cong \tw d\, \Sigma(q)$ in \cite[Theorem 2.5.12]{GLS} is greater than one, then \eqref{cond-o} holds (for this choice of $J$). It remains to consider the untwisted groups, i.e.
the ones with $d=1$. 

\smallskip
If $\bar{S}=\PSL_2(q)$ then \eqref{cond-o} holds. 
Suppose that $\bar{S}=\PSL_n(q)$ with $n \geq 3$, or $\bar{S}=P\Omega^+_{2n}(q)$ with $n \geq 4$, or 
$S=E_6(q)$. Then taking $B:= \bigl(\bar{G} \cap \mathrm{Inndiag}(\bar{S})\bigr)/\bar{S}$, we see that 
$A/B$ is cyclic by assumption (b) in Theorem \ref{main2}, hence \eqref{cond-a} holds.

In the remaining cases, $\bar{S}$ is of type $B_n$, $C_n$, $G_2$, $F_4$, $E_7$, or $E_8$, hence $O/J$ is cyclic, and so \eqref{cond-o} holds.
\end{proof}

Next we deduce another consequence of Theorem \ref{main1}. For the definition of the {\it reduced Clifford group} $\Gamma^+(\FF_q^n)=\mathrm{CSpin}^\epsilon_n(q)$, 
see e.g. \cite[\S6]{TZ}; in particular, it contains $\Spin^\epsilon_n(q)$ as a normal subgroup with factor $C_{q-1}$.

\begin{theor}\label{main3}
Let $q$ be a prime power, and let $(G,S)$ be any of the following pairs of groups:
\begin{enumerate}[\rm(a)]
\item $G=\GL_n(q)$ with $n \geq 2$, $(n,q) \neq (2,2)$, $(2,3)$, and $S=\SL_n(q)$.
\item $G=\GU_n(q)$ with $n \geq 2$, $(n,q) \neq (2,2)$, $(2,3)$, $(3,2)$, and $S=\SU_n(q)$.
\item $G=\mathrm{CSp}_{2n}(q)$ with $n \geq 2$, $(n,q) \neq (2,2)$, and $S=\Sp_{2n}(q)$.
\item $G=\mathrm{CSpin}^\epsilon_n(q)$ with $n \geq 5$, $2 \nmid q$, and $\epsilon=\pm$, and $S=\mathrm{Spin}^\epsilon_n(q)$.
\item $G=\GO^\epsilon_n(q)$ or $\SO^\epsilon_n(q)$ with $n \geq 5$, $2 \nmid q$, and $\epsilon=\pm$, and $S=\Omega^\epsilon_n(q)$.
\end{enumerate}
Suppose that $a,b \in G$ are such that $p \nmid |\CB_S(a)|$ and $p \nmid |\CB_S(b)|$. If $g \in G$ is any non-central $p'$-element such 
that $g \in abS$, then $g \in a^Gb^G$.
\end{theor}

\begin{proof}
For all of the above pairs, we have $S$ is a quasisimple group of Lie type of simply connected type,
$S \lhd G$, $F^*(G)=\ZB(G)S$, and $\ZB(S) \leq \ZB(G)$. Furthermore, $G/S$ is abelian of $p'$-order, and
\eqref{for-a}, \eqref{for-b}, and \eqref{for-g} are all fulfilled. Hence the statement follows from Theorem \ref{main1}, unless we are in case (e). In
the latter case, the same proof of Theorem \ref{main1} applies.
\end{proof}

Note that Theorem \ref{main1} also applies to $\GO^\epsilon_{2n}(q)$ with $2|q$ and $n \geq 3$. But we do not include them in Theorem 
\ref{main3} since the subgroup $D$ is now of order $2$ and so conditions \eqref{for-a}--\eqref{for-g} do not have the simplest form as formulated in 
Theorem \ref{main3}.

\end{document}